\documentclass[11pt]{article}
\usepackage[latin1]{inputenc}
\usepackage{amsmath,amsthm,amssymb}
\usepackage{amsfonts}
\usepackage{amsmath,amsthm,amssymb,amscd}
\usepackage{latexsym}
\usepackage{color}
\usepackage[shortlabels]{enumitem}
\usepackage{graphicx}
\usepackage{enumitem}
\usepackage{mathrsfs}
\usepackage{mathtools}
\usepackage{blindtext}

\makeatletter \@addtoreset{equation}{section} \makeatother

\newtheorem{theorem}{Theorem}[section]
\newtheorem{definition}{Definition}[section]

\newtheorem{lemma}{Lemma}[section]
\newtheorem{remark}{Remark}[section]

\begin{document}	
	
	
	%\title{Existence of infinitely many solutions}
	
	%\begin{frontmatter}
		
		
		%\title{Existence of infinitely many solutions}
		
		\title{\bf  \sc		
		On elliptic problems with Choquard term  and  singular nonlinearity}
		%\tnotetext[label1]{}
		
		\author{\sc Debajyoti Choudhuri$^{a}$, Du\v{s}an D.  Repov\v{s}$^{b,}$\footnote{Corresponding author%
		: dusan.repovs@guest.arnes.si}, Kamel Saoudi$^c$ \\
	\small{$^{a}$Department of Mathematics, National Institute of Technology Rourkela,}\\
	\small{Rourkela, 769008, Odisha, India.	{\it Email: dc.iit12@gmail.com}}\\
	\small{$^b$Faculty of Education and Faculty of Mathematics and Physics, University of Ljubljana,}\\
	\small{\& Institute of Mathematics, Physics and Mechanics, Ljubljana, 1000, Slovenia. {\it Email: dusan.repovs@guest.arnes.si}}\\	
	\small{$^c$ Basic and Applied Scientific Research Center, Imam Abdulrahman Bin Faisal University,}\\
	\small{Dammam 34212, Saudi Arabia. {\it Email: kmsaoudi@iau.edu.sa}}\\}			

\date{}

\maketitle
				\begin{abstract}
					\noindent Using variational  methods, we establish the existence of infinitely many solutions to an elliptic problem driven by a Choquard term and a singular nonlinearity. We further show that if the problem has a positive solution, then it is bounded a.e. in the domain $\Omega$ and is H\"{o}lder continuous.
					\begin{flushleft}
						{\it Math. Subj. Classif.}:~35R35, 35Q35, 35J20, 46E35.
					\end{flushleft}
				\end{abstract}
				
				\section{Introduction}
				We shall study the following problem:
				%\begin{align}
				%\begin{split}
				\begin{equation}\label{main_1}
				\left\{ \begin{array}{ll}
				-\Delta u+V(x)u=\mu(J_{\alpha}*|u|^p)|u|^{p-2}u+\lambda |u|^{-\beta-1}u~\text{in}~\Omega,\\
				u=0 \ \ \ \ \ \ \ \ \ \ \ \ \ \ \ \ \ \ \ \ \ \ \ \ \ \ \ \ \ \ \ \ \ \ \ \ \ \ \ \ \ \ \ \  \ \ \ \ \ \ \ \ \ \ \ \ \ \text{on}~\partial\Omega,
				\end{array} 
				\right.
				\end{equation}
				%\end{split}
				%\end{align}
				where 
				$$0<\beta<1, 0<\alpha<N, \mu\in\mathbb{R}, \lambda>0, p\in[2_{\alpha},2_{\alpha}^*], \  \ 2_{\alpha}=\frac{N+\alpha}{N}, 2_{\alpha}^{*}=\frac{N+\alpha}{N-2}, $$
				and $V$ is a continuous function (popularly called the {\it potential well}), satisfying the following conditions
				\begin{description}
					\item[$(V_1)$] $V:\Omega\to\mathbb{R}$ is continuous and there exists a constant $V_0>0$ such that $V(x)>V_0~\text{for all}~x\in\Omega$. Here, $N\geq 2$ and $\Omega\subset\mathbb{R}^N$ is a bounded domain,
					\item[$(V_2)$] $V:\Omega\to\mathbb{R}$ satisfies the integrability condition $$\int_{\Omega}\frac{1}{V(x)}dx<\infty.$$
				\end{description}
				The symbol $(*)$ in problem \eqref{main_1} denotes the convolution. We note that both $2_{\alpha} $ and $ 2_{\alpha}^*$ are less than $2^*=\frac{2N}{N-2}$ which is the Sobolev critical exponent.
				The Riesz potential can be described as $$J_{\alpha}(x)=\frac{A_{\alpha}(N)}{|x|^{N-\alpha}}, \quad A_{\alpha}(N)=\frac{\Gamma(N-\alpha/2)}{\Gamma(\alpha/2)\pi^{N/2}2^{\alpha}}.$$ 
				\\
				\noindent
				There are many articles pertaining to elliptic nonlocal problems driven by the Choquard term with various assumptions on the potential function $V$. For example, the readers may refer to {\sc Moroz-Schaftingen} \cite{moroz1}, {\sc Carvalho et al.} \cite{car1}, {\sc Mukherjee-Sreenadh} \cite{muk1}, and the references therein. The consideration of a upper critical Choquard term with singularity and Radon measure, albeit in a bounded domain, can be found in {\sc Panda et al} \cite{panda1}, who established the existence of positive solutions.\\				
				 A systematic study of elliptic problems driven by singular nonlinearity began four decades ago with the seminal work by {\sc Lazer-McKenna} \cite{LM}. This gave a new direction for research to the field of elliptic PDEs. We refer to some of the important works that answer the question of existence of solution to singularity driven elliptic problems in {\sc Giacomoni et al.} \cite{gia1}, {\sc Saoudi et al.} \cite{ghosh1}, and the references therein.
				
				\subsection{Significance of the considered problem}
				Problem \eqref{main_1} addresses a wider class of problem in the sense that we have considered $\mu\in\mathbb{R}$. This is an improvement of the existing works in which the authors required the coefficient of a Choquard term to be strictly positive. One can see that if $\lambda=0$, $\mu=1$, $V(x)=c$, $N=3$ in problem \eqref{main_1}, then we have 
				\begin{align}\label{lions}
				\begin{split}
				-\Delta u+cu-\left(\int_{\mathbb{R}^3}u^2(y)V(x-y)dy\right)u(x)=0.
				\end{split}
				\end{align}
				Here $V$ is a positive function. This problem was studied by {\sc Lions} in \cite{lions1} for an unbounded domain $\mathbb{R}^3$. So the study of problem \eqref{main_1} is not very new and has been of interest for quite a long time.\\
				Furthermore, when $V(x)\equiv 0$, $\mu=0$, the problem reduces to the classical problem similar to the one considered by {\sc Lazer-McKenna} \cite{LM}. However, a combination of the Choquard term and a singular term with $\mu\in\mathbb{R}$ is a new idea. The main results of our paper are stated below.
				\begin{theorem}\label{MP_thm}
					For any $\lambda>0$ and $\mu\in\mathbb{R}$, the functional $I_{\lambda,\mu}$ has a sequence of critical points $(u_n)$ such that $I_{\lambda,\mu}(u_n)<0$ and $I_{\lambda,\mu}(u_n)\to 0$, as $n\to\infty$.
				\end{theorem}
				\noindent
			Our second main result establishes the boundedness of positive weak solutions of \eqref{main_1}. To the best of our knowledge, this was unknown for elliptic PDEs involving a Choquard term and a singular nonlinearity. 
				\begin{theorem}\label{boundedness_of_soln}
					Let $0<u\in X$ be a weak solution of problem \eqref{main_1}. Then  $u\in L^{\infty}(\Omega)$. 
				\end{theorem}
				
				Our third main result shows that such positive solutions are not only bounded but are also H\"{o}lder continuous regular.
				
				\begin{theorem}\label{reg}
					Let $u>0$ be a weak solution of problem \eqref{main_1}. Then $u\in C_{\text{loc}}^{\theta}(\Omega)$ for some $\theta\in(0,1)$.
				\end{theorem}
				\noindent
				We complete this section with a description of the structure of our paper. 
				In Section $2$, the mathematical preliminaries are discussed. 
				In Section $3$, Theorem 1.1 is proved. 
				In Section $4$, Theorem 1.2 is proved. 
				In Section $5$, Theorem 1.3 is proved. 
				We end the paper in Section 6, with an interesting application of the Choquard term will be presented.
			
				\section{Preliminaries}
				
				We shall seek for a solution in the space $X$ which is defined as follows:
				$$X=\{v\in H_0^{1,2}(\Omega):\int_{\Omega}V(x)v^2dx<\infty\}.$$
				and endowed with the norm
				$$\|v\|=\left(\int_{\Omega}(|\nabla v|^2+V(x)v^2)dx\right)^{1/2}.$$
				We define the inner product on $X$ as follows:
				\begin{align}\label{innerproduct}
				\langle\phi,\psi\rangle&=\int_{\Omega}(\nabla\phi\cdot\nabla\psi+V(x)\phi\psi)dx~\text{where}~\phi, \psi\in X.
				\end{align}
				A consequence of the well known embedding result is that $X\hookrightarrow L^r(\Omega)$, for each $r\in[1,2^*]$. Likewise, the embedding $X\hookrightarrow L^r(\Omega)$ is compact for $r\in[1,2^*)$. \\
			The following is the energy functional corresponding to problem \eqref{main_1}:
				\begin{align}\label{energy_func}
				I(u)&=\frac{1}{2}\|u\|^2-\frac{\mu}{2p}\int_{\Omega}(J_{\alpha}*(u^+(x))^p)(u^+(x))^pdx-\frac{\lambda}{1-\beta}\int_{\Omega}(u^+(x))^{1-\beta}dx,~u\in X.
				\end{align}
				However, $I$ is not a $C^1$ functional and hence the results from the premise of the variational methods cannot be used. However, we shall overcome this difficulty by defining a {\it cutoff} energy function as follows:
				$$\mathcal{C}(s)=\begin{cases}
				1, ~\text{if}~ |s|\leq l\\
				\eta ~\text{is decreasing, if}~ l\leq s\leq 2l\\
				0,~\text{if}~ |s|\geq 2l.
				\end{cases}$$
				Clearly, $0\leq\mathcal{C}(s)\leq 1$ for every $s\in\mathbb{R}$. Accordingly, we consider the following the modified problem
				\begin{align*}
				\begin{split}
				(\mathcal{P}):~~~~~~~~-&\Delta u+V(x)u=\mu g(u)+\lambda |u|^{-\beta-1}u~\text{in}~\Omega\\
								u&=0~\text{on}~\partial\Omega.
				\end{split}
				\end{align*}
				Here,
				\begin{align*}
				\begin{split}
				g(u(x))&=\{(J_{\alpha}*|u(x)|^p)|u(x)|^{p-2}u\}\mathcal{C}\left(\frac{\|u\|^2}{2}\right),\\
				G(u(x))&=\frac{1}{2p}\{(J_{\alpha}*|u(x)|^p)|u(x)|^{p}\}\mathcal{C}\left(\frac{\|u\|^2}{2}\right).
				\end{split}
				\end{align*}
				The corresponding energy functional is thus defined as
				\begin{align}\label{cutoff_energy_func}
				I_{\lambda,\mu}(u)&=\frac{1}{2}\|u\|^2-\frac{\mu}{2p}\int_{\Omega}(J_{\alpha}*|u(x)|^p)|u(x)|^pdx-\lambda\int_{\Omega}G(u(x))dx,~u\in X.
				\end{align}
		
				Corresponding to $I_{\lambda,\mu}$, we define a modified version of problem \eqref{main_1} as follows:

				\begin{definition}\label{WS}
					A function $u\in X$ is said to be a {\it weak solution} to $(\mathcal{P})$ if 
					\begin{align}\label{WS_Defn}
					\begin{split}
					\langle u,\phi\rangle-\mu\int_{\Omega}(J_{\alpha}*|u(x)|^p)|u(x)|^{p-2}u\phi(x) dx-\frac{\mu}{2p}\mathcal{C}'\left(\frac{\|u\|^2}{2}\right)\langle u,\phi \rangle\int_{\Omega}(J_{\alpha}*|u(x)|^p)|u(x)|^{p}\phi(x) dx\\
					-\lambda\int_{\Omega}|u|^{-\beta-1}u\phi(x) dx=0~\text{for every}~\phi\in X.
					\end{split}
					\end{align}
				\end{definition}
				
				\noindent We observe that if $\|u\|\leq\sqrt{2l}$ and obeys the identity in \eqref{WS_Defn}, then $u$ also is a weak solution to problem \eqref{main_1}. Henceforth, we shall use the word solution instead of weak solution.\\
			We also refer to an important inequality called the Hardy-Littlewood-Sobolev (HLS) inequality  (see  \cite[Lemma $2.1$]{moroz1}) which is usually used in the literature to estimate the Choquard term in the literature.\\
				In the course of the proof of the main result, we shall need the best Sobolev constant which is defined by
				$$S:=\underset{u\in X\setminus\{0\}}{\inf}\frac{\int_{\Omega}|\nabla u|^2dx}{\left(\int_{\Omega}\int_{\Omega}\frac{|u(x)|^{2_{\alpha}^*}|u(y)|^{2_{\alpha}^*}}{|x-y|^{N-\alpha}}dxdy\right)^{1/2_{\alpha}^*}}.$$
				For all other preliminary information we refer the reader to the comprehensive monograph by {\sc Papageorgiou et al.} \cite{PPR}.
			
				\section{Proof of Theorem \ref{MP_thm}}
				The first step involved in variational methods to verify whether a functional has a critical point or not is to see if the functional exhibits a mountain pass geometry or not. The following lemma verifies this.
				\begin{lemma}\label{MP_geometry}
					There exists $r>0$ such that $\underset{\|u\|=r}{\inf}{I}_{\lambda,\mu}(u)>0.$
				\end{lemma}
				\begin{proof}
					We begin by the following estimates
					\begin{align}\label{est1}
					I_{\lambda,\mu}(u)&=\frac{1}{2}\|u\|^2-\frac{\mu}{2p}\int_{\Omega}(J_{\alpha}*|u(x)|^p)|u(x)|^pdx-\lambda\int_{\Omega}G(u(x))dx\\
					&\geq \frac{1}{2}\|u\|^2-\frac{\mu}{2p}\int_{\Omega}(J_{\alpha}*|u(x)|^p)|u(x)|^pdx-\lambda\int_{\Omega}G(u(x))dx\\
					&\geq \frac{1}{2}\|u\|^2-\frac{C\mu}{2p}\|u\|^{2p}-C'\lambda\|u\|^{1-\beta},
					\end{align}
					where the Sobolev embedding result has been used along with the Hardy-Littlewood-Sobolev inequality. Therefore, for a sufficiently small choice of positive numbers $\lambda, r:=\|u\|$, we have $I_{\lambda}(u)>0$ for any $\|u\|=r$.
				\end{proof}
				We shall employ the {\it dual fountain theorem} due to {\sc Bartsch-Willem} (see \cite[Theorem $3.18$]{BW_thm}) to show that there are infinitely many critical points of the functional $I_{\lambda,\mu}$. 
				\begin{proof}[Proof of Theorem \ref{MP_thm}]
					Let $(e_n)$ be a sequence of orthonormal basis of $X$ and let $X_n:=\{ae_n:a\in\mathbb{R}\}$. We shall
					use the following subspaces in the proof:
					\begin{align*}
					\begin{split}
					Y_n:=\underset{0\leq k\leq j}{\oplus}X_k,~~Z_k:=\overline{\underset{j\leq k<\infty}{\oplus}X_k}.
					\end{split}
					\end{align*}
					We further define 
					$$B_n:=\underset{u\in Z_n, \|u\|=1}{\sup}\int_{\Omega}|u|^{1-\beta}dx.$$
					We notice that for a small enough $l>0$, $0<R<l/2<1$, one has
					\begin{align}\label{FT1}
					|\mu|\frac{C}{2p}\|u\|^{2p}\leq\frac{1}{4}\|u\|^2.
					\end{align}
					We further observe that $0<B_{n+1}\leq B_n$ for each $n$, hence $B_n\to B$ as $n\to\infty$. By the definition of $B_n$ we have for every $n\geq 1$, $u_n\in Z_n$, $\|u_n\|=1$, such that $$\int_{\Omega}|u_n|^{1-\beta}dx>\frac{B_n}{2}.$$ 
					Also, since $\|u_n\|=1$, we have $u_n\rightharpoonup 0$ as $n\to\infty$. By the embedding theorem, we can conclude that $B_n\to 0$ as $n\to\infty$.\\
					Thus by \eqref{FT1}, for all $u\in Z_n$, $\|u\|\leq R$, we have 
					\begin{align}\label{est2}
					\begin{split}
					I_{\lambda,\mu}(u)&\geq\frac{1}{2}\|u\|^2-C\frac{|\mu|}{2p}\|u\|^{2p}-B_n^{1-\beta}\frac{\lambda}{1-\beta}\|u\|^{1-\beta}\\
					&\geq\frac{1}{4}\|u\|^2-B_n^{1-\beta}\frac{\lambda}{1-\beta}\|u\|^{1-\beta}.
					\end{split}
					\end{align}
					We define $\rho_n:=(4\lambda B_n^{1-\beta})^{\frac{1}{1+\beta}}\to 0$ as $n\to\infty$. Thus, there exists $n_0$ such that $\|u\|\leq R$ for any $n\geq n_0$. Hence, for any $n\geq n_0$, $u\in Z_n$, $\|u\|=\rho_n$, we have $I_{\lambda,\mu}(u)\geq 0$.\\
					By \eqref{est2} we have 					
					\begin{align}\label{PS_0}
					\begin{split}
					I_{\lambda,\mu}(u)\geq-\lambda B_n^{1-\beta}\|u_n\|^{1-\beta}/(1-\beta)\geq-\lambda B_n^{1-\beta}\rho_n^{1-\beta}/(1-\beta)
					\end{split}
					\end{align}
					for any $n\geq n_0$, $u\in Z_n$, $\|u\|\leq\rho_n$. Also, since  for any $n\geq n_0$  the $\rho_n$'s are sufficiently small, and $\|u\|<\rho_n$, we have $I_{\lambda,\mu}(u)<0$. Since $B_n\to 0$, $\rho_n\to 0$ as $n\to\infty$, it follows that 
					$$D_n:=\underset{u\in Z_n,\|u\|\leq\rho_n}{\inf}I_{\lambda,\mu}(u)\to 0$$
					as $n\to\infty$.  In fact, we name $D_{n_0}$ to be the first member of the sequence $(D_n)$ after which eventually all the $D_n$'s are negative.\\
					Now, since any two norms are topologically equivalent in a finite dimensional space $Y_n$, there exists a sufficiently small $r_n>0$, for $\lambda>0$, such that \begin{align}\label{FT2}\underset{u\in Y_n, \|u\|=r_n}{\max}I_{\lambda,\mu}(u)<0.\end{align}\\
					We now show that the functional $I_{\lambda,\mu}$ obeys the Palais-Smale condition. We shall 
					first devise a mechanism to select a suitable sequence in such a way that the functional $I_{\lambda,\mu}$ is still well defined. Let $(u_n)\subset X$ be an eventually zero sequence, then it converges to $0$. We shall discard it as the functional may not be defined and moreover such sequences anyway converge to zero. Furthermore, if $(u_n)\subset X$ is a sequence with infinitely (or finitely) many terms being equal to zero, then we choose a subsequence with only nonzero terms of $(u_n)$. Therefore, without loss of generality, let $(u_n)$ be a sequence such that $u_n\neq 0$ for every $n\in\mathbb{N}$ and also which does not converge to zero in $X$ as $n\rightarrow \infty$. We divide the proof into three cases:\\
					{\it Case 1}:~ Suppose $p\in(2_{\alpha},2_{\alpha}^*)$. Consider
					\begin{align}\label{PS''}
					\begin{split}
					c+\sigma\|u_n\|+o(1)&=I_{\lambda,\mu}(u_n)-\frac{1}{2p}\langle I_{\lambda,\mu}'(u_n),u_n\rangle\\
					&=\left(\frac{1}{2}-\frac{1}{2p}\right)\|u_n\|^2-\lambda\left(\frac{1}{1-\beta}-\frac{1}{2p}\right)\int_{\Omega}|u_n|^{1-\beta}dx\\
					&\geq \left(\frac{1}{2}-\frac{1}{2p}\right)\|u_n\|^2-\lambda B_0^{1-\beta}\left(\frac{1}{1-\beta}-\frac{1}{2p}\right)\|u_n\|^{1-\beta}.
					\end{split}
					\end{align}
					This implies that $(u_n)$ is bounded and hence there exists a subsequence, still denoted as $(u_n)$, such that $u_n\rightharpoonup u$ in $X$. By the weak convergence $u_n\rightharpoonup u$ and the embedding theorem we have $\langle I'_{\lambda,\mu}u,u\rangle=0$. Furthermore, we consider the following
					\begin{align*}
					\begin{split}
					o(1)=&\langle I'_{\lambda,\mu}u_n,u_n-u\rangle\\
					=&\|u_n\|^2-\|u\|^2- \mu\int_{\Omega}(J_{\alpha}*|u_n|^{p})|u_n|^{p}dx+\mu\int_{\Omega}(J_{\alpha}*|u|^{p})|u|^{p}dx\\
					&-\lambda\int_{\Omega}|u_n|^{1-\beta}dx+\lambda\int_{\Omega}|u_n|^{-\beta-1}u_nudx=\|u_n\|^2-\|u\|^2.
					\end{split}
					\end{align*}
					We have again used a combination of the embedding theorems and the Egorov's theorem above. Thus $u_n\to u$ in $X$ and hence the $(PS)$ condition is satisfied by the functional $I_{\lambda,\mu}$.\\
					{\it Case 2}:~ Suppose $p=2_{\alpha}^*$. We shall
					show that $I_{\lambda,\mu}$ satisfies the Palais-Smale $(PS)_c$ condition for energy level $$c<c^*:=\frac{1}{2}\left(\frac{\alpha+2}{N+\alpha}\right)S^{\frac{N-2}{\alpha+2}}-\frac{1}{2}\left(\frac{N+\alpha}{\alpha+2}\right)^{\frac{1+\beta}{1-\beta}}\left(\lambda C\left(\frac{1}{1-\beta}-\frac{1}{2\cdot 2_{\alpha}^*}\right)\right)^{\frac{2}{1+\beta}}~\text{for any}~\lambda>0.$$
					We again consider $(u_n)\subset X$ such that $I_{\lambda,\mu}(u_n)\to c<c^*$, $I'_{\lambda,\mu}(u_n)\to 0$. Suppose $\underset{n\to\infty}{\lim}\|u_n-u\|^2=l^2>0$. Then
					\begin{align}\label{PS}
					\begin{split}
					c+\sigma\|u_n\|+o(1)&=I_{\lambda,\mu}(u_n)-\frac{1}{2\cdot 2_{\alpha}^*}\langle I_{\lambda,\mu}'(u_n),u_n\rangle\\
					&=\left(\frac{1}{2}-\frac{1}{2\cdot 2_{\alpha}^*}\right)\|u_n\|^2-\lambda\left(\frac{1}{1-\beta}-\frac{1}{2\cdot 2_{\alpha}^*}\right)\int_{\Omega}(u_n^+)^{1-\beta}dx\\
					&\geq \left(\frac{1}{2}-\frac{1}{2\cdot 2_{\alpha}^*}\right)\|u_n\|^2-\lambda B_0^{1-\beta}\left(\frac{1}{1-\beta}-\frac{1}{2\cdot 2_{\alpha}^*}\right)\|u_n\|^{1-\beta}.
					\end{split}
					\end{align}
					It follows from \eqref{PS} that the sequence $(u_n)$ is bounded in $X$. Therefore by the reflexivity of $X$, there exists a subsequence, still denoted by $(u_n)$, such that $u_n\rightharpoonup u$ as $n\to\infty$. It is straight forward to show that $u$ is a weak solution to the modified problem $(\mathcal{P})$ and hence is a weak solution to the problem \eqref{main_1}, i.e. $\langle I'_{\lambda,\mu}(u),u\rangle=0$.\\ 
					Furthermore, invoking the weak convergence of $(u_n)$, the Brezis-Lieb lemma, the embedding theorems and the Egorov theorem we can make the following observation
					\begin{align}\label{lim_0}
					\begin{split}
					o(1)=&\langle I_{\lambda,\mu}(u_n),u_n\rangle\\
					=&\|u_n\|^2-\mu\int_{\Omega}(J_{\alpha}*|u_n|^{2_{\alpha}^*})|u_n|^{2_{\alpha}^*}dx-\lambda\int_{\Omega}|u_n|^{1-\beta}dx\\
					=&\|u_n-u\|^2+\|u\|^2-\mu\int_{\Omega}(J_{\alpha}*|u_n-u|^{2_{\alpha}^*})|u_n-u|^{2_{\alpha}^*}dx+\mu\int_{\Omega}(J_{\alpha}*|u|^{2_{\alpha}^*})|u|^{2_{\alpha}^*}dx\\
					&-\lambda\int_{\Omega}|u|^{1-\beta}dx\\
					=&\langle I_{\lambda,\mu}(u),u\rangle+\|u_n-u\|^2-\mu\int_{\Omega}(J_{\alpha}*|u_n-u|^{2_{\alpha}^*})|u_n-u|^{2_{\alpha}^*}dx\\
					=&\|u_n-u\|^2-\mu\int_{\Omega}(J_{\alpha}*|u_n-u|^{2_{\alpha}^*})|u_n-u|^{2_{\alpha}^*}dx.
					\end{split}
					\end{align}  
					Let 
					$$\underset{n\to\infty}{\lim}\|u_n-u\|^2=\underset{n\to\infty}{\lim}\int_{\Omega}(J_{\alpha}*|u_n-u|^{2_{\alpha}^*})|u_n-u|^{2_{\alpha}^*}dx=l.$$
					From the definition of the best Sobolev constant given in Section $2$, we have 
					$$S\leq\frac{l}{l^{1/{2_{\alpha}^*}}}.$$
					Hence, by an application of the Young inequality in \eqref{PS} we get
					\begin{align}
					\begin{split}
					c+o(1)\geq&\left(\frac{1}{2}-\frac{1}{2\cdot 2_{\alpha}^*}\right)(l^2+\|u_n\|^2)-\left(\frac{1}{2}-\frac{1}{2\cdot 2_{\alpha}^*}\right)\|u_n\|^2\\
					&-\left(\frac{1}{2}-\frac{1}{2\cdot 2_{\alpha}^*}\right)^{-\frac{1-\beta}{1+\beta}}\left(\lambda C\left(\frac{1}{1-\beta}-\frac{1}{2\cdot 2_{\alpha}^*}\right)\right)^{\frac{2}{1+\beta}}\\
					\geq&\left(\frac{1}{2}-\frac{1}{2\cdot 2_{\alpha}^*}\right)S^{2.\frac{N+\alpha}{\alpha+2}}-\left(\frac{1}{2}-\frac{1}{2\cdot  2_{\alpha}^*}\right)^{-\frac{1-\beta}{1+\beta}}\left(\lambda C\left(\frac{1}{1-\beta}-\frac{1}{2\cdot 2_{\alpha}^*}\right)\right)^{\frac{2}{1+\beta}}.
					\end{split}
					\end{align}
					This is a violation to the range of energy level considered in the hypothesis of the theorem. Hence, $$\underset{n\to\infty}{\lim}\|u_n-u\|^2=0.$$\\
					{\it Case 3}:~The obeying of $(PS)_c$ condition by the functional $I_{\lambda,\mu}$ for the case of $p=2_{\alpha}$ follows in a similar manner as in {\it Case 2}. In fact it holds for $$c<c^{**}:=\frac{1}{2}\left(\frac{\alpha}{N+\alpha}\right)S^{2\frac{N+\alpha}{\alpha}}-\frac{1}{2}\left(\frac{N+\alpha}{\alpha}\right)^{\frac{1+\beta}{1-\beta}}\left(\lambda C\left(\frac{1}{1-\beta}-\frac{1}{2p}\right)\right)^{\frac{2}{1+\beta}}~\text{for any}~\lambda>0.$$
					Therefore for the critical cases of $p=2_{\alpha}, 2_{\alpha}^*$, the $(PS)$ condition holds for $c<\min\{c^*,c^{**}\}$ for any $\lambda>0$.\\
					Since all the hypothesis of the dual fountain theorem are satisfied, we conclude the existence of a sequence of negative critical values that converge to $0$. Hence $I_{\lambda,\mu}$ has infinite number of solutions and consequently the \eqref{main_1} has infinitely many solutions.
				\end{proof}			
		
				\section{Proof of Theorem \ref{boundedness_of_soln}	}
				We shall prove that all nonnegative solutions of the problem \eqref{main_1} are bounded. The sketch of the proof is as follows (it runs along the lines of the argument in \cite{muk1}).
				\begin{proof}[Proof of Theorem \ref{boundedness_of_soln}]
					Let $u\in X$ be a nonnegative solution with $|\{x\in\Omega:u(x)=0\}|=0$. We shall
					prove the boundedness for the case of $p=2_{\alpha^*}$. The cases when $p\in(2_{\alpha},2_{\alpha}^*)$ and $p=2_{\alpha}$ follow the same argument. We define $w(x):=\frac{u(x)}{d}$ which, for a suitable $d>0$, obeys
					\begin{align}\label{ineq1}
					\begin{split}
					\int_{\Omega}\nabla & u\cdot\nabla\phi dx+\int_{\Omega}V(x) u\phi dx\\
					\leq&|\mu|\int_{\Omega}(J_{\alpha}*(v(x))^{2_{\alpha}^*})(v(y))^{2_{\alpha}^*-2}v(y)\phi(x)dx
					+\lambda\int_{\Omega}(v)^{-\beta}\phi dx
					\end{split}
					\end{align}
					for every nonnegative $\phi\in C_c^{\infty}(\Omega)$. We further define the following:
					\begin{align}\label{symbols}
					\begin{split}
					D_n:=1-3^{-n},\
					w_n:=w-D_n,\
					v_n:=w_n^{+},\
					V_n:=\|v_n\|_{2^*}.
					\end{split}
					\end{align}
					Thus $0\leq |w|+D_n\leq|w|+1$. Of course $|w|+1\in L^{2}(\Omega)\subset L^{2^*}(\Omega)$ since $\Omega$ is bounded. Moreover, $\underset{n\to\infty}{\lim}v_n=(w-1)^+$. By the dominated convergence theorem we get 
					\begin{align}\label{ineq2}
					\underset{n\to\infty}{\lim} V_n&=\left(\int_{\Omega}[(w-1)^+]^{2^*dx}\right)^{1/2^*}.
					\end{align}
					From the definitions above we make the following observation:
					\begin{align}\label{obs1}
					\begin{split}
					D_n<D_{n+1},~\text{which implies}\
					v_{n+1}\leq v_n~\text{a.e. in}~\Omega.
					\end{split}
					\end{align}
					Furthermore,  we define 
					$$B_n:=\frac{D_{n+1}}{D_{n+1}-D_n}=(1/2)(3^{n+1}-1)<(3^{n+1}-1)
					\
					\hbox{ for all} 
					\
					n\in\mathbb{N}.$$
					We make a similar observation as in \cite{muk1} that
					\begin{align}\label{ineq3}w< B_nv_n~\text{on}~\{v_{n+1}>0\}.\end{align}
					Employing \eqref{ineq1}, \eqref{ineq3} in tandem with H\"{o}lder's inequality, we obtain
					\begin{align}\label{ineq4}
					\begin{split}
					\int_{\Omega}|\nabla v_{n+1}|^2dx\leq&\int_{\Omega}\nabla v_{n+1}\cdot\nabla v_{n+1}^+
					\leq\int_{\Omega}\nabla w\cdot\nabla v_{n+1}^+dx\\
					\leq&B_n^{2_{\alpha}^*-1}|\mu|\int_{\{v_{n+1}>0\}}\int_{\Omega}\frac{(v(y))^{2_{\alpha}^*}|v_{n}(x)|^{2_{\alpha}^*-1}v_{n+1}(x)}{|x-y|^{N-\alpha}}dydx+\lambda B_n\int_{\{v_{n+1}>0\}}|v_n|^2dx\\
					\leq&B_n^{2_{\alpha}^*-1}|\mu|\int_{\{v_{n+1}>0\}}\int_{\Omega}\frac{(v(y))^{2_{\alpha}^*}|v_{n}(x)|^{2_{\alpha}^*}}{|x-y|^{N-\alpha}}dydx+\lambda B_n\|v_{n}\|_{2^*}^{1-\beta}|\{v_{n+1}>0\}|^{\frac{N+2+\beta(N-2)}{2N}}.
					\end{split}
					\end{align}
					We estimate the first integral on the right hand side by dividing it into a sum of two integrals as follows:
					\begin{align}\label{ineq5}
					\begin{split}
					\int_{\{v_{n+1}>0\}}&\int_{\Omega}\frac{(v(y))^{2_{\alpha}^*}|v_{n}(x)|^{2_{\alpha}^*}}{|x-y|^{N-\alpha}}dydx\\
					&\leq \left(\int_{\{v_{n+1}>0\}}\int_{\{v(y)\leq D_{n+1}\}}+\int_{\{v_{n+1}>0\}}\int_{\{v(y)>D_{n+1}\}}\right)\frac{(v(y))^{2_{\alpha}^*}|v_{n}(x)|^{2_{\alpha}^*}}{|x-y|^{N-\alpha}}dydx:=\mathcal{I}_1+\mathcal{I}_2.
					\end{split}
					\end{align}
					The HLS inequality and the H\"{o}lder inequality yields
					\begin{align}\label{ineq6}
					\begin{split}
					\mathcal{I}_1\leq&B_n^{2_{\alpha}^*}C(N,\alpha)\|v_n\|_{2^*}^{2\cdot2_{\alpha}^*}\\
					\mathcal{I}_2\leq&CB_{n+1}^{2_{\alpha}^*}|\{v_{n+1}>0\}|^{(N-\alpha)/2N}\|v_{n}\|_{2^*}^{2_{\alpha}^*}.
					\end{split}
					\end{align}
					The embedding theorem in combination with \eqref{ineq4} and \eqref{ineq6} yields the following:
					\begin{align}\label{ineq7}
					\begin{split}
					S\|v_{n+1}\|_{2^*}^{2}\leq& B_n^{2_{\alpha}^*-1}\left(\mu B_n^{2_{\alpha}^*}C(N,\alpha)\|v_n\|_{2^*}^{2\cdot 2_{\alpha}^*}+ \mu CD_{n+1}^{2_{\alpha}^*}|\{v_{n+1}>0\}|^{(N-\alpha)/2N}\|v_n\|_{2^*}^{2_{\alpha}^*}\right.\\
					&\left.+\lambda B_n3^{n+1}\|v_n\|_{2^*}^{1-\beta}|\{v_{n+1}>0\}|^{\frac{N+2+\beta(N-2)}{2N}}\right),
					\end{split}
					\end{align}
					where $$S=\underset{u\in X\setminus\{0\}}{\inf}\frac{\int_{\Omega}|\nabla u|^2dx}{(\int_{\Omega}|u|^{2^*}dx)^{2/2^*}}.$$\\
					Again, by following the proof in \cite{muk1}, we get that
					\begin{align}\label{ineq8}
					\{v_{n+1}>0\}&\subset\{v_n>1/3^{n+2}\}.
					\end{align}
					Therefore by \eqref{ineq8} we get
					\begin{align}\label{ineq9}
					V_n^{2^*}&=\|v_n\|_{2^*}\geq\int_{\{v_n>3^{-n-2}\}}v_n^{2^*}dx\geq 3^{-n-2}|\{v_{n+1}>0\}|.
					\end{align}
			and
					\begin{align}\label{ineq10}
					\begin{split}
					S\|v_{n+1}\|_{2^*}^{2}\leq& B_n^{2_{\alpha}^*-1}\left(\mu B_n^{2_{\alpha}^*}C(N,\alpha)\|v_n\|_{2^*}^{2\cdot 2_{\alpha}^*}+ \mu CD_{n+1}^{2_{\alpha}^*}3^{(n+2)(N-\alpha)/2N}\|v_n\|_{2^*}^{2^*}\right.\\
					&\left.+\lambda B_n3^{n+2}\|v_n\|_{2^*}^{2^*}2^{(n+1)\frac{N+2+\beta(N-2)}{2N}}\right)
					\leq3^{(2_{\alpha}^*-1)(n+2)}\mathcal{C}(\|v_n\|_{2^*}^{2\cdot 2_{\alpha}^*}+\|v_n\|_{2^*}^{2^*}),
					\end{split}
					\end{align}
					where $$\mathcal{C}:=\max\{3^{2_{\alpha}^*(n+2)}C(N,\alpha),3^{(n+2)(N-\alpha)/2N}C+3^{(n+2)(3N+2+\beta(N-2))/2N}\}.$$
					Thus by the definition of $V_n$ in \eqref{symbols} we have
					\begin{align}\label{ineq11}
					V_{n+1}=&\mathcal{D}^{n+2}(V_n^{2_{\alpha}^*}+V_n^{2^*/2}),
					\end{align}
					where $\mathcal{D}=(1+\sqrt{\mathcal{C}^{1/n+2}})>1$. We choose $\epsilon>0$ sufficiently small so that 
					\begin{align}\label{ineq12}
					\epsilon^{2/N-2}&<\frac{1}{(3^{2_{\alpha}^*}\mathcal{D})^{N-2/2}}.
					\end{align}
					On fixing $\gamma\in\left(\epsilon^{2/N-2},\frac{1}{3^{2_{\alpha}^*}\mathcal{D}^{N-2/2}}\right)$, we conclude that $\gamma\in (0,1)$ since $N/N-2,\mathcal{D}>1$. Also 
					\begin{align}\label{ineq13}
					\epsilon^{2/N-2}\leq\gamma, 3^{2_{\alpha}^*}\mathcal{D}\gamma^{N/N-2}\leq 1.
					\end{align}
					It can be proved by induction that 
					\begin{align}\label{ineq14}
					V_{n+1}\leq& 2\epsilon\gamma^{n+1}.
					\end{align}
					From the fact that $\gamma\in(0,1)$ and by \eqref{ineq14} we see that $\underset{n\to\infty}{\lim}V_n=0$. Thus by \eqref{ineq2}, we have $(w-1)^+=0$ a.e. in $\Omega$. This implies that $0\leq w\leq 1$ a.e. in $\Omega$. Hence, $0\leq u\leq d$ a.e. in $\Omega$. Therefore $\|u\|_{\infty}\leq c$.
				\end{proof}
				
				\section{Proof of Theorem \ref{reg}}
				Finally, we prove a regularity result for the positive solutions of problem \eqref{main_1}.
				
				\begin{proof}[Proof of Theorem \ref{reg}]
					Let $\tilde{\Omega}\Subset\Omega$. Thus for any $\phi\in C_c^{\infty}(\Omega)$, we have, by the boundedness of $u$, the following:
					\begin{align*}\begin{split}
					\int_{\Omega}(\nabla u\cdot\nabla\phi+V(x)u\phi)dx=&\mu\int_{\tilde{\Omega}}(J_{\alpha}*(u(x))^{2_{\alpha}^*})(u(x))^{2_{\alpha}^*-2}u(x)\phi(x)dx
					+\lambda\int_{\tilde{\Omega}}u^{-\beta}\phi dx\\
					\leq&|\mu|\int_{\tilde{\Omega}}(J_{\alpha}*(u(x))^{2_{\alpha}^*})(u(x))^{2_{\alpha}^*-2}u(x)\phi(x)dx
					+\lambda\int_{\tilde{\Omega}}u^{-\beta}\phi dx\\
					\leq& C\int_{\tilde{\Omega}}\phi dx.
					\end{split}
					\end{align*}
					Here we have used the fact that $u\geq C_{K}>0$ a.e. for any compact $K\subset\Omega$. For if not, then the integral with the singular term ceases to exists. Therefore, $(-\Delta)u+V(x)u$ is weakly bounded in $\tilde{\Omega}$. \\
					By {\sc Iannizzotto et al.} \cite[Theorem $4.4$]{ian1} and the covering argument applied to the inequality in \cite[Corollary $5.5$ ]{ian1} we produce a $\theta\in (0,1)$ such that $u\in C^{\theta}(\tilde{\Omega})$, for any $\tilde{\Omega}\Subset\Omega$. Hence $u\in C_{\text{loc}}^{\theta}(\tilde{\Omega})$.
				\end{proof}
				\noindent We end this section with the following remark.
				\begin{remark}\label{neg_solns}
					We observe that if $u>0$ is a solution to \eqref{main_1}, then $-u<0$ is also a solution which is bounded and H\"{o}lder continuous.
				\end{remark}
				
				\section{Application of the Choquard term}
				In his seminal work, {\sc Lieb} \cite{lieb1} pointed out that Choquard used this term to approximate the Hartree-Fock theory of component plasma. For the case  $$N=3, \alpha=2, p=2, \lambda=0, V(x)\equiv 0,$$
				problem \eqref{main_1} was used by {\sc Peker} \cite{peker} to study the quantum theory of a polaron at rest.
				The Choquard term also found its application in multiparticle system courtesy {\sc Gross} \cite{gross} and in quantum mechanics {\sc Penrose} \cite{penrose}.\\
				\noindent
				For each solution of problem \eqref{main_1}, we define $\phi(t,x)=e^{it}u(x)$ which is called the wave function, where $i=\sqrt{-1}$. Here $\phi$ is a solitary wave of the focussing time-dependent Hartree equation:
				$$-i\phi_t-\Delta \phi+W(x)\phi-(J_{\alpha}*|\phi|^{p})|\phi|^{p-2}\phi-\lambda|u|^{q-2}u=0~\text{in}~\mathbb{R}\times\mathbb{R}^N,$$
				where $W(x)=V(x)-1$, $x\in\mathbb{R}^N$. Thus  problem \eqref{main_1} can be interpreted as the stationary nonlinear Hartree equation.
				
				\subsection*{Acknowledgements}
				The first author acknowledges the funding received from the Science and Engineering Research Board (SERB), India, under the MATRICS scheme [grant number MTR/2018/000525].
				The second author acknowledges the funding received from the Slovenian Research Agency [P1-0292, J1-4031, J1-4001, N1-0278, N1-0114, N1-0083].
				
			\end{document}